\documentclass[12pt,leqno]{amsart}
\usepackage{euscript, amssymb, amsmath, amsthm}
\usepackage{epsfig}
\usepackage{graphicx}
\usepackage{caption}
\usepackage{longtable}
\usepackage{dcolumn}
\usepackage{setspace}
\usepackage[most]{tcolorbox}
\definecolor{webred}{rgb}{0.75,0,0}
\definecolor{webgreen}{rgb}{0,0.75,0}
\definecolor{refkey}{gray}{0.75}
\usepackage[pagebackref=true, colorlinks=true, citecolor=blue]{hyperref}

\numberwithin{equation}{section}

\newtheorem{theo}{Theorem}[section]
\newtheorem{lem}{Lemma}[section]

\newtheorem{Def}[theo]{Definition}
\theoremstyle{remark}
\newtheorem{rem}{Remark}[section]

\newcommand{\ep}{\varepsilon}
\def\R{{\mathbb{R}}}

\def\d{\displaystyle}
\def\e{{\varepsilon}}
\def\p{\partial}

\date{}


\tcbset{
    frame code={}
    center title,
    left=0pt,
    right=0pt,
    top=0pt,
    bottom=0pt,
    colback=gray!10,
    colframe=white,
    width=\dimexpr\textwidth\relax,
    enlarge left by=0mm,
    boxsep=5pt,
    arc=0pt,outer arc=0pt,
    }

\begin{document}

\title[Improvement on blow-up of the solution of a damped wave equation - Invariant case]{Improvement on the blow-up of the wave equation  with the  scale-invariant damping and combined nonlinearities}
\author[M. Hamouda  and M. A. Hamza]{Makram Hamouda$^{1}$ and Mohamed Ali Hamza$^{1}$}
\address{$^{1}$ Basic Sciences Department, Deanship of Preparatory Year and Supporting Studies, P. O. Box 1982, Imam Abdulrahman Bin Faisal University, Dammam, KSA.}

\medskip

\email{mmhamouda@iau.edu.sa (M. Hamouda)} 
\email{mahamza@iau.edu.sa (M.A. Hamza)}

\pagestyle{plain}


\maketitle
\begin{abstract}
We consider in this article the  damped wave equation,  in the \textit{scale-invariant case} with combined two nonlinearities, which reads as follows:
\begin{displaymath}
\d (E) \hspace{1cm} u_{tt}-\Delta u+\frac{\mu}{1+t}u_t=|u_t|^p+|u|^q,
\quad \mbox{in}\ \R^N\times[0,\infty),
\end{displaymath}
with small initial data.\\
Compared to our previous work \cite{Our}, we show in this article that the first hypothesis on the damping coefficient $\mu$, namely $\mu < \frac{N(q-1)}{2}$, can be removed, and the second  one  can be extended from $(0,  \mu_*/2)$ to $(0,  \mu_*)$ where $\mu_*>0$ is solution of $(q-1)\left((N+\mu_*-1)p-2\right) = 4$.  Indeed, owing to a better understanding of the influence of the damping term in the global dynamics of the solution, we think that this new interval for $\mu$  characterizes better the threshold between the blow-up and the global existence regions. Moreover, taking advantage of the techniques employed in the problem $(E)$, we also improve the result in \cite{LT2,Palmieri}  in relationship with the Glassey conjecture for the solution of $(E)$ without the nonlinear term $|u|^q$. More precisely, we extend the blow-up region from $p \in (1, p_G(N+\sigma)]$, where $\sigma$ is given by \eqref{sigma} below, to $p \in (1, p_G(N+\mu)]$ giving thus a better estimate of the lifespan in this case.
\end{abstract}


\section{Introduction}
\par\quad

We consider the following  family of semilinear damped wave equations
\begin{equation}
\label{G-sys}
\left\{
\begin{array}{l}
\d u_{tt}-\Delta u+\frac{\mu}{1+t}u_t=a|u_t|^p+b|u|^q,
\quad \mbox{in}\ \R^N\times[0,\infty),\\
u(x,0)=\e f(x),\ u_t(x,0)=\e g(x), \quad  x\in\R^N,
\end{array}
\right.
\end{equation}
where $a$ and $b$ are nonnegative constants and $\mu \ge 0$. Moreover, the parameter $\e$ is a positive number describing the size of the initial
data,  and $f$ and $g$ are positive functions which are compactly supported on  $B_{\R^N}(0,R), R>0$.

Throughout this article, we suppose that $p, q>1$ and $q \le \frac{2N}{N-2}$ if $N \ge 3$.

The corresponding linear equation  to \eqref{G-sys} is given by
\begin{equation}\label{1.2}
\d u^L_{tt}-\Delta u^L+\frac{\mu}{1+t}u^L_t=0.
\end{equation} 
It is well-known that the equation \eqref{1.2} is invariant under the following transformation:
$$\tilde{u}^L(x,t)=u^L(\Omega x, \Omega(1+t)-1), \ \Omega>0.$$
The above scaling justifies the designation of the \textit{scale-invariant} case for \eqref{G-sys}. It is interesting to recall that the scale-invariant case is the interface between the class of parabolic equations (for $\mu$ large enough) and the one of hyperbolic equations (for small values of $\mu$). In fact, in this transition, the parameter $\mu$ plays a crucial role in determining the behavior of the solution of \eqref{1.2}, see for example \cite{Wirth1}.\\

Let $\mu = 0$ and 
$(a,b)=(0,1)$ in \eqref{G-sys}, then the equation \eqref{G-sys} is thus  the classical semilinear wave equation for which we have the Strauss conjecture. This case  is characterized by a critical power, denoted by $q_S$,  which is solution of the following quadratic equation
\begin{equation}
(N-1)q^2-(N+1)q-2=0,
\end{equation}
and is given  by
\begin{equation}
q_S=q_S(N):=\frac{N+1+\sqrt{N^2+10N-7}}{2(N-1)}.
\end{equation}
More precisely,  if $q \le q_S$ then there is no global solution for  \eqref{G-sys} under suitable sign assumptions for the initial data, and for $q > q_S$ a global solution exists for small initial data; see e.g. \cite{John2,Strauss,YZ06,Zhou} among many other references. \\

Now,  the case $\mu = 0$  and $(a,b)=(1,0)$ is obeying to the Glassey conjecture which asserts that the critical power $p_G$ should be given by
\begin{equation}\label{Glassey}
p_G=p_G(N):=1+\frac{2}{N-1}.
\end{equation}
The above critical value, $p_G$, gives rise to two regions for the power $p$ ensuring the global existence (for $p>p_G$) or the nonexistence (for $p \le p_G$) of a global small data solution; see e.g. \cite{Hidano1,Hidano2,John1,Rammaha,Sideris,Tzvetkov,Zhou1}.\\

It is worth-mentioning that the presence of two nonlinearities in \eqref{G-sys} has an interesting effect on the existence or the nonexistence of global in-time solution of \eqref{G-sys} and its lifespan. In fact, compared to a one single nonlinearity,  the presence of mixed nonlinearities produces an additional new blow-up region. 

First, we focus on the case  $\mu = 0$ and $a, b \neq 0$, thus, without loss of generality we may assume that $(a,b)=(1,1)$. It is easy to see that  in this case, together with the assumption that  the powers $p$ and $q$ satisfy   $p \le p_G$ or $q \le q_S$, the blow-up of the solution of  \eqref{G-sys} can be handled in a similar way. Therefore, for $p > p_G$ and $q > q_S$, the new blow-up border  is characterized by the following relationship between $p$ and $q$:
\begin{equation}\label{1.5}
\lambda(p, q, N):=(q-1)\left((N-1)p-2\right) < 4.
\end{equation}
We refer the reader to \cite{Dai,Han-Zhou,Hidano3,Wang} for more details.\\
Note that it is proven in \cite{Hidano3} that,  for $p > p_G$ and $q > q_S$, the equality in \eqref{1.5} yields the global existence of the solution of \eqref{G-sys} (with $\mu = 0$ and $(a,b)=(1,1)$) without going through the intermediate step of ``almost global solution''. This is in fact related to the presence of mixed  nonlinearities. Naturally, it is interesting to see whether this phenomenon still occurs for the damping case $\mu > 0$. 

Now, we focus on the case $\mu > 0$. In fact, for $(a,b)=(0,1)$, it is known in the literature that if the  weak damping coefficient $\mu$ is relatively large, then the equation \eqref{G-sys} (with $(a,b)=(0,1)$) behaves like the corresponding heat equation; see e.g. \cite{dabbicco1,dabbicco2,wakasugi}. Though, if  $\mu$ is small, then the behavior of \eqref{G-sys} is following the one of the corresponding wave equation. More precisely, for $\mu$ small, it was proven, in 
\cite{LTW} and later on  in \cite{Ikeda} with a substantial improvement, that the dimension in the critical power $q_S$ is shifted by $\mu>0$  compared to the one in the case without damping ($\mu=0$), and hence we have  for 
$$0<\mu < \frac{N^2+N+2}{N+2} \quad \text{and} \quad 1<q\le q_S(N+\mu),$$
the blow-up of the solution of \eqref{G-sys}. These blow-up results have been improved in many ways in \cite{Palmieri-Reissig, Palmieri, Palmieri-Tu, Tu-Lin1, Tu-Lin}. \\
In the particular case $\mu=2$ and $N=2,3$, the same above observation is proven in \cite{Dab2}, see also \cite{Dab1}. For the global existence in this case ($\mu=2$) we refer the reader to \cite{Dab1, Dab2, Palmieri2}.\\

On the other hand for $\mu > 0$ and $(a,b)=(1,0)$, the authors prove in \cite{LT2} a blow-up result for the solution of \eqref{G-sys} (with $(a,b)=(1,0)$) 
and they give an upper bound of the lifespan. We stress the fact that in this case there is no restriction  for $\mu$  in the blow-up region for $p$, namely $p \in (1, p_G(N+2\mu)]$. Recently, Palmieri and Tu  proved in \cite{Palmieri}, among many other interesting results, a more accurate blow-up interval for $p$ in relationship with the solution of  \eqref{G-sys} with $(a,b)=(1,0)$, one time-derivative nonlinearity (that is \eqref{T-sys-bis} below) and a mass term. More precisely, it is proven that the solution of this problem blows up in finite time for $p \in (1, p_G(N+\sigma)]$ where 
\begin{equation}\label{sigma}
\sigma=\left\{
\begin{array}{lll}
2 \mu & \textnormal{if} &\mu \in [0,1),\\
2 & \text{if} &\mu \in [1,2),\\
\mu & \text{if} &\mu \ge 2.
\end{array}
\right.
\end{equation}
Of course the problems studied  in \cite{Palmieri} are more general, but, we want to point out here  the improvement obtained for \eqref{T-sys-bis} below.\\

In this article, we are interested in the study of the following Cauchy problem which is related to
the scale-invariant wave equation with combined nonlinearities. More precisely, we consider
\begin{equation}
\label{T-sys}
\left\{
\begin{array}{l}
\d u_{tt}-\Delta u+\frac{\mu}{1+t}u_t=|u_t|^p+|u|^q, 
\quad \mbox{in}\ \R^N\times[0,\infty),\\
u(x,0)=\e f(x),\ u_t(x,0)=\e g(x), \quad  x\in\R^N,
\end{array}
\right.
\end{equation}
where $\mu > 0$, $\ N \ge 1$, $\e>0$ is a sufficiently  small parameter
and  $f,g$ are positive functions chosen in the energy space with compact support.

The emphasis in our  work is the study of the Cauchy problem (\ref{T-sys}) for $\mu>0$ and the influence of the parameter $\mu$ on the blow-up result and the lifespan estimate. For the analogous system of (\ref{T-sys}) with  $(\mu/(1+t))u_t$ being replaced by $(\mu/(1+t)^\beta)u_t$ and  $\beta >1$, which corresponds to the scattering case,  Lai and  Takamura proved in \cite{LT3} that, comparing to the wave equation without damping, the scattering damping term has no influence in the dynamics.  However, in the scale-invariant case ($\beta =1)$ the situation is different. Indeed,  as expected, the combination of a weak damping term and the two mixed nonlinearities  is playing here a crucial role. This fact has been shown in our previous work \cite{Our} where the focus were on the obtaining of the equation of the hyperbola part for the blow-up region. In the present work, owing to a better comprehension of the role of the weak damping term in (\ref{T-sys}) in the dynamics and following the technique used in \cite{Tu-Lin} (the method is based on the use of some test functions which closely describe the solution of the linear part of  (\ref{T-sys})), we will here improve the bound of the blow-up region in the hyperbola part. We precisely show in this article that \eqref{1.5} holds for 
$\lambda(p, q, N+\mu)$ instead of $\lambda(p, q, N+2\mu)$ which was obtained in  \cite{Our}. Hence, the new hypothesis on $p$ and $q$ ($\lambda(p, q, N+\mu)<4$) constitutes a shift by $\mu$ of the  dimension $N$. Naturally, we obtain a better bound for the lifespan. We face here a situation similar to the Strauss exponent case, $q_S$, as explained above.  However,  thanks to a good choice of the functional as in \eqref{G-exp} below, we succeed in this article to remove one of the smallness  hypotheses on $\mu$, namely $\mu < \frac{N(q-1)}{2}$,  which was assumed in \cite[Theorem 2.3]{Our}. \\

Furthermore, keeping in mind all the hypotheses as for (\ref{T-sys}), we consider here the following equation with only one time-derivative nonlinearity, namely 
\begin{equation}
\label{T-sys-bis}
\left\{
\begin{array}{l}
\d u_{tt}-\Delta u+\frac{\mu}{1+t}u_t=|u_t|^p, 
\quad \mbox{in}\ \R^N\times[0,\infty),\\
u(x,0)=\e f(x),\ u_t(x,0)=\e g(x), \quad  x\in\R^N.
\end{array}
\right.
\end{equation}
Taking advantage of the techniques used for (\ref{T-sys}), we will improve the blow-up interval, $p \in (1, p_G(N+\sigma)]$ ($\sigma$ is given by \eqref{sigma}),  obtained in \cite{Palmieri}, which is itself an improvement of \cite{LT2},  to  reach the interval $p \in (1, p_G(N+\mu)]$, for $\mu \in (0,2)$. However, for $\mu \ge 2$, our result for (\ref{T-sys-bis}) coincides  with the one in \cite{Palmieri}. We may conjecture that the obtained upper bound exponent is the critical one in the sense that it gives the threshold between the blow-up and the global existence regions. \\

The  article is organized as follows. We start in Section \ref{sec-main} by introducing  the weak formulation of (\ref{T-sys}) in the energy space. Then, we state the main theorems of our work.  Some technical lemmas are thus proven in Section \ref{aux}. These auxiliary results, among other tools, are used to conclude  the proof of the main results in Sections \ref{proof} and \ref{sec-ut}.  More precisely, in Section \ref{proof} (resp. Sec. \ref{sec-ut}), we prove that the solution of (\ref{T-sys}) (resp. (\ref{T-sys-bis})) for $p$ and $q$ satisfying $\lambda(p, q, N+\mu)<4$ (resp. for $p$ verifying $p \in (1, p_G(N+\mu)]$) blows up in finite time.

\section{Main Results}\label{sec-main}
\par

This section is aimed to state our main results. For that purpose, we first start by giving the definition of the solution of (\ref{T-sys}) in the corresponding energy space. More precisely, the weak formulation associated with  (\ref{T-sys}) reads  as follows:
\begin{Def}\label{def1}
 We say that $u$ is a weak  solution of
 (\ref{T-sys}) on $[0,T)$
if
\begin{displaymath}
\left\{
\begin{array}{l}
u\in \mathcal{C}([0,T),H^1(\R^N))\cap \mathcal{C}^1([0,T),L^2(\R^N)), \vspace{.1cm}\\
 u \in L^q_{loc}((0,T)\times \R^N) \ \text{and} \ u_t \in L^p_{loc}((0,T)\times \R^N),
 \end{array}
  \right.
\end{displaymath}
satisfies, for all $\Phi\in \mathcal{C}_0^{\infty}(\R^N\times[0,T))$ and all $t\in[0,T)$, the following equation:
\begin{equation}
\label{energysol2}
\begin{array}{l}
\d\int_{\R^N}u_t(x,t)\Phi(x,t)dx-\int_{\R^N}u_t(x,0)\Phi(x,0)dx \vspace{.2cm}\\
\d -\int_0^t  \int_{\R^N}u_t(x,s)\Phi_t(x,s)dx \,ds+\int_0^t  \int_{\R^N}\nabla u(x,s)\cdot\nabla\Phi(x,s) dx \,ds\vspace{.2cm}\\
\d  +\int_0^t  \int_{\R^N}\frac{\mu}{1+s}u_t(x,s) \Phi(x,s)dx \,ds=\int_0^t \int_{\R^N}\left\{|u_t(x,s)|^p+|u(x,s)|^q\right\}\Phi(x,s)dx \,ds.
\end{array}
\end{equation}
\end{Def}

Obviously, the weak formulation corresponding to (\ref{T-sys-bis}) can be also given by \eqref{energysol2} with simply ignoring the nonlinear term $|u|^q$ with the necessary modifications accordingly.

The following theorems state the  main results of this article.
\begin{theo}
\label{blowup}
Let $p, q$ and $\mu >0$  such that 
\begin{equation}\label{assump}
\lambda(p, q, N+\mu)<4,
\end{equation}
where the expression of $\lambda$ is given by \eqref{1.5}, and $p>p_G(N+\mu)$ and $q>q_S(N+\mu)$.\\
Assume that  $f\in H^1(\R^N)$ and $g\in L^2(\R^N)$ are non-negative functions which are compactly supported on  $B_{\R^N}(0,R)$,
and  do not vanish everywhere.
Let $u$ be an energy solution of \eqref{T-sys} on $[0,T_\e)$ such that $\mbox{\rm supp}(u)\ \subset\{(x,t)\in\R^N\times[0,\infty): |x|\le t+R\}$. 
Then, there exists a constant $\e_0=\e_0(f,g,N,R,p,q,\mu)>0$
such that $T_\e$ verifies
\[
T_\e\leq
 C \,\e^{-\frac{2p(q-1)}{4-\lambda(p, q, N+\mu)}},
\]
 where $C$ is a positive constant independent of $\e$ and $0<\e\le\e_0$.
\end{theo}

\begin{theo}
\label{th_u_t}
Let $\mu>0$. Assume that  $f\in H^1(\R^N)$ and $g\in L^2(\R^N)$ are non-negative functions which are compactly supported on  $B_{\R^N}(0,R)$,
and  do not vanish everywhere.
Let $u$ be an energy solution of \eqref{T-sys-bis} on $[0,T_\e)$ such that $\mbox{\rm supp}(u)\ \subset\{(x,t)\in\R^N\times[0,\infty): |x|\le t+R\}$. 
Then, there exists a constant $\e_0=\e_0(f,g,N,R,p,\mu)>0$
such that $T_\e$ verifies
\begin{displaymath}
T_\e \leq
\d \left\{
\begin{array}{ll}
 C \, \e^{-\frac{2(p-1)}{2-(N+\mu-1)(p-1)}}
 &
 \ \text{for} \
 1<p<p_G(N+\mu), \vspace{.1cm}
 \\
 \exp\left(C\e^{-(p-1)}\right)
&
 \ \text{for} \ p=p_G(N+\mu),
\end{array}
\right.
\end{displaymath}
 where $C$ is a positive constant independent of $\e$ and $0<\e\le\e_0$.
\end{theo}

\begin{rem}
It is well-known that the assumption \eqref{1.5} with $p > p_G(N)$ and $q > q_S(N)$ still yields the blow-up of the corresponding undamped equation to \eqref{T-sys} (with $\mu=0$), see \cite{Han-Zhou,Hidano3,LT3}. More precisely, it is proven in \cite[Remark 2.3]{LT3} the existence of a pair $(p_0(N),q_0(N))$ verifying  \eqref{1.5}, $p_0(N) > p_G(N)$ and $q_0(N) > q_S(N)$. Consequently, for $\mu>0$, we have  $(p_0(N+\mu),q_0(N+\mu))$ which satisfy  \eqref{assump}, $p_0(N+\mu) > p_G(N+\mu)$ and $q_0(N+\mu) > q_S(N+\mu)$. Hence, the hypothesis on $p$ and $q$ in Theorem \ref{blowup} makes sense.
\end{rem}

\begin{rem}
Theorem \ref{th_u_t} asserts that the critical exponent for $p$ is greater than $p_G(N+\mu)$. We believe that the limiting value $p_G(N+\mu)$  coincides with the critical one. Of course one has to rigorously confirm this assertion. This will be the subject of a forthcoming work.
\end{rem}

\begin{rem}
Unlike the case with only one nonlinearity ($|u_t(x,s)|^p$ or $|u(x,s)|^q$), one can note, in addition to the two blow-up regions $p \le p_G(N)$ and $q \le q_S(N)$, the obtaining  of another blow-up region, characterized by \eqref{1.5} together with  $p>p_G(N)$ and $q>q_S(N)$.  The new  region is in fact due to the presence of two mixed nonlinearities in \eqref{T-sys}, see \cite{Hidano3} for the problem \eqref{T-sys} with $\mu=0$. This observation still holds in our case but with \eqref{1.5} being replaced by \eqref{assump}, otherwise $p_G(N)$ being replaced by $p_G(N+ \mu)$ and $q_S(N)$  by $q_S(N+\mu)$, for $\mu$ small. It was previously conjectured that  $q_S(N+\mu)$ constitutes the critical value for $q$ between the blow-up and the global existence zones, and Theorem \ref{blowup} gives a first assertion for this conjecture.   To complete the whole picture, the global existence in-time of the solution of \eqref{T-sys} will be studied in a subsequent work.
\end{rem}

\begin{rem}
The assumption \eqref{assump} can be seen as a smallness condition for $\mu$, namely
$\mu \in [0, \mu_*)$ where $\mu= \mu_*$ satisfies  the equality in \eqref{assump}  (otherwise $\mu_*:=\frac{2(q+1)}{p(q-1)} -N+1$). 
\end{rem}

\begin{rem}
Note that the results in 
Theorems \ref{blowup} and  \ref{th_u_t} hold true after replacing the linear damping term in \eqref{T-sys}  $\frac{\mu}{1+t} u_t$ by $b(t)u_t$ with $[b(t)-\mu (1+t)^{-1}]$ belongs to $L^1(0,\infty)$. The proof of this  generalized damping case can be obtained by following the same steps as in the proofs of Theorems \ref{blowup}  and  \ref{th_u_t} with the necessary modifications.
\end{rem}

\section{Some auxiliary results}\label{aux}
\par

We define the following positive test function  
\begin{equation}
\label{test11}
\psi(x,t):=\rho(t)\phi(x);
\quad
\phi(x):=
\left\{
\begin{array}{ll}
\d\int_{S^{N-1}}e^{x\cdot\omega}d\omega & \mbox{for}\ N\ge2,\vspace{.2cm}\\
e^x+e^{-x} & \mbox{for}\  N=1,
\end{array}
\right.
\end{equation}
where $\phi(x)$ is introduced in \cite{YZ06}   and $\rho(t)$, \cite{Palmieri1,Palmieri-Tu,Tu-Lin1,Tu-Lin},   is solution of 
\begin{equation}\label{lambda}
\frac{d^2 \rho(t)}{dt^2}-\rho(t)-\frac{d}{dt}\left(\frac{\mu}{1+t}\rho(t)\right)=0.
\end{equation}
The expression of  $\rho(t)$ reads as follows (see the Appendix for more details):
\begin{equation}\label{lmabdaK}
\rho(t)=(t+1)^{\frac{\mu+1}{2}}K_{\frac{\mu-1}2}(t+1),
\end{equation}
where 
$$K_{\nu}(t)=\int_0^\infty\exp(-t\cosh \zeta)\cosh(\nu \zeta)d\zeta,\ \nu\in \mathbb{R}.$$
Moreover, the function $\phi(x)$ verifies
\begin{equation*}\label{phi-pp}
\Delta\phi=\phi.
\end{equation*}
Note that the function $\psi(x, t)$ satisfies the conjugate equation corresponding to \eqref{1.2}, namely we have
\begin{equation}\label{lambda-eq}
\partial^2_t \psi(x, t)-\Delta \psi(x, t) -\frac{\partial}{\partial t}\left(\frac{\mu}{1+t}\psi(x, t)\right)=0.
\end{equation}

Throughout this article, we will denote by $C$  a generic positive constant which may depend on the data ($p,q,\mu,N,R,f,g$) but not on $\ep$ and whose the value may change from line to line. Nevertheless, we will precise the dependence of the constant $C$ on the parameters of the problem when it is necessary.\\

The following lemma  holds true for the function $\psi(x, t)$.
\begin{lem}[\cite{YZ06}]
\label{lem1} Let  $r>1$.
There exists a constant $C=C(N,R,p,r)>0$ such that
\begin{equation}
\label{psi}
\int_{|x|\leq t+R}\Big(\psi(x,t)\Big)^{r}dx
\leq C\rho^r(t)e^{rt}(1+t)^{\frac{(2-r)(N-1)}{2}},
\quad\forall \ t\ge0.
\end{equation}
\end{lem}
\par
As in the non-perturbed case, we define here the functionals that we will use to prove the blow-up criteria later on:
\begin{equation}
\label{F1def}
G_1(t):=\int_{\R^N}u(x, t)\psi(x, t)dx,
\end{equation}
and
\begin{equation}
\label{F2def}
G_2(t):=\int_{\R^N}\p_tu(x, t)\psi(x, t)dx.
\end{equation}
The next two lemmas give the first  lower bounds for $G_1(t)$ and $G_2(t)$, respectively. More precisely, we will prove that $G_1(t)$ and $G_2(t)$ are two {\it coercive} functions. This is the first observation which will be used to improve the main results of this article. Indeed, in our previous work \cite{Our}, and compared to the results in \eqref{F1postive} and \eqref{F2postive} below, we obtained weaker lower bounds for the functionals, $G_1(t)$ and $G_2(t)$, of size $\e / (1+t)^{\mu/2}$ (instead of $\e$ here); see Lemmas 3.2 and 3.3 in \cite{Our}.

We note here that the proof of Lemma \ref{F1} is known in the literature; see e.g. \cite{Palmieri1,Tu-Lin1,Tu-Lin}. However, we choose to include all the details about the proof of this lemma, on the one hand, to make the article self-contained and, on the other hand,  to make later use of some computations therein. Nevertheless, Lemma \ref{F11} constitutes a novelty in this work and its utilization in the proofs of Theorems \ref{blowup}  and  \ref{th_u_t} is fundamental.  
\begin{lem}
\label{F1}
Assume the existence of an energy solution $u$ of the system \eqref{T-sys}  with initial data satisfying the assumptions in Theorem \ref{blowup}. Then, there exists $T_0=T_0(\mu)>1$ such that 
\begin{equation}
\label{F1postive}
G_1(t)\ge C_{G_1}\, \e, 
\quad\text{for all}\ t \ge T_0,
\end{equation}
where $C_{G_1}$ is a positive constant which depends on $f$, $g$, $N,R$ and $\mu$.
\end{lem}
\begin{proof} 
Let $t \in [0,T)$.  Using Definition \ref{def1} and  performing an integration by parts in space in the fourth term in the left-hand side of \eqref{energysol2}, we obtain
\begin{equation}\label{eq3}
\begin{array}{l}
\d \int_{\R^N}u_t(x,t)\Phi(x,t)dx
-\e\int_{\R^N}g(x)\Phi(x,0)dx \vspace{.2cm}\\
\d-\int_0^t\int_{\R^N}\left\{
u_t(x,s)\Phi_t(x,s)+u(x,s)\Delta\Phi(x,s)\right\}dx \, ds +\int_0^t  \int_{\R^N}\frac{\mu}{1+s}u_t(x,s) \Phi(x,s)dx \,ds\vspace{.2cm}\\
\d=\int_0^t\int_{\R^N}\left\{|u_t(x,s)|^p+|u(x,s)|^q\right\}\Phi(x,s)dx \, ds, \quad \forall \ \Phi\in \mathcal{C}_0^{\infty}(\R^N\times[0,T)).
\end{array}
\end{equation}
Now, substituting in \eqref{eq3} $\Phi(x, t)$ by $\psi(x, t)$, we infer that
\begin{equation}
\begin{array}{l}\label{eq4}
\d \int_{\R^N}u_t(x,t)\psi(x,t)dx
-\e\int_{\R^N}g(x)\psi(x,0)dx \vspace{.2cm}\\
\d-\int_0^t\int_{\R^N}\left\{
u_t(x,s)\psi_t(x,s)+u(x,s)\Delta\psi(x,s)\right\}dx \, ds +\int_0^t  \int_{\R^N}\frac{\mu}{1+s}u_t(x,s) \psi(x,s)dx \,ds\vspace{.2cm}\\
\d=\int_0^t\int_{\R^N}\left\{|u_t(x,s)|^p+|u(x,s)|^q\right\}\psi(x,s)dx \, ds.
\end{array}
\end{equation}
Performing an integration by parts for the first and third terms in the second line of \eqref{eq4} and utilizing \eqref{test11} and \eqref{lambda-eq}, we obtain
\begin{equation}
\begin{array}{l}\label{eq5}
\d \int_{\R^N}\big[u_t(x,t)\psi(x,t)- u(x,t)\psi_t(x,t)+\frac{\mu}{1+t}u(x,t) \psi(x,t)\big]dx
\vspace{.2cm}\\
\d=\int_0^t\int_{\R^N}\left\{|u_t(x,s)|^p+|u(x,s)|^q\right\}\psi(x,s)dx \, ds 
+\d \e \, C(f,g),
\end{array}
\end{equation}
where 
\begin{equation}\label{Cfg}
C(f,g):=\rho(0)\int_{\R^N}\big[\big(\mu-\frac{\rho'(0)}{\rho(0)}\big)f(x)\phi(x)+g(x)\phi(x)\big]dx.
\end{equation}
We notice  that the constant $C(f,g)$ is positive thanks to \eqref{lambda'lambda} and the fact that the function $K_{\nu}(t)$ is positive (see \eqref{Kmu} in the Appendix). \\
Hence, using the definition of $G_1$, as in \eqref{F1def},  and \eqref{test11}, the equation  \eqref{eq5} yields
\begin{equation}
\begin{array}{l}\label{eq6}
\d G_1'(t)+\Gamma(t)G_1(t)=\int_0^t\int_{\R^N}\left\{|u_t(x,s)|^p+|u(x,s)|^q\right\}\psi(x,s)dx \, ds +\e \, C(f,g),
\end{array}
\end{equation}
where 
\begin{equation}\label{gamma}
\Gamma(t):=\frac{\mu}{1+t}-2\frac{\rho'(t)}{\rho(t)}.
\end{equation}
Multiplying  \eqref{eq6} by $\frac{(1+t)^\mu}{\rho^2(t)}$ and integrating over $(0,t)$, we deduce  that
\begin{align}\label{est-G1}
 G_1(t)
\ge G_1(0)\frac{\rho^2(t)}{(1+t)^\mu}+{\e}C(f,g)\frac{\rho^2(t)}{(1+t)^\mu}\int_0^t\frac{(1+s)^\mu}{\rho^2(s)}ds.
\end{align}
Using \eqref{lmabdaK} and  the fact that $G_1(0)>0$, the estimate \eqref{est-G1} yields
\begin{align}\label{est-G1-1}
 G_1(t)
\ge {\e}C(f,g)(1+t)K^2_{\frac{\mu-1}2}(t+1)\int^t_{t/2}\frac{1}{(1+s)K^2_{\frac{\mu-1}2}(s+1)}ds.
\end{align}
From \eqref{Kmu}, we have the existence of $T_0=T_0(\mu)>1$ such that 
\begin{align}\label{est-double}
(1+t)K^2_{\frac{\mu-1}2}(t+1)>\frac{\pi}{4} e^{-2(t+1)} \quad \text{and}  \quad (1+t)^{-1}K^{-2}_{\frac{\mu-1}2}(t+1)>\frac{1}{\pi} e^{2(t+1)}, \ \forall \ t \ge T_0/2.
\end{align}
Hence, we have
\begin{align}\label{est-G1-2}
 G_1(t)
\ge \frac{\e}{4}C(f,g)e^{-2t}\int^t_{t/2}e^{2s}ds\ge \frac{\e}{8}C(f,g)e^{-2t}(e^{2t}-e^{t}), \ \forall \ t \ge T_0.
\end{align}
Finally, using $e^{2t}>2e^{t}, \forall \ t \ge 1$, we deduce that
\begin{align}\label{est-G1-3}
 G_1(t)
\ge \frac{\e}{16}C(f,g), \ \forall \ t \ge T_0.
\end{align}

This ends the proof of Lemma \ref{F1}.
\end{proof}

Now we are in a position to prove  the following lemma.
\begin{lem}\label{F11}
For any energy solution $u$ of the system \eqref{T-sys} with initial data satisfying the assumptions in Theorem \ref{blowup},  there exists $T_1=T_1(\mu)>0$ such that 
\begin{equation}
\label{F2postive}
G_2(t)\ge C_{G_2}\, \e, 
\quad\text{for all}\ t  \ge  T_1,
\end{equation}
where $C_{G_2}$ is a positive constant which depends on $f$, $g$, $N$ and $\mu$.
\end{lem}
 
\begin{proof}
Let $t \in [0,T)$. Then, using the definition of $G_1$ and  $G_2$, given respectively by \eqref{F1def} and  \eqref{F2def}, \eqref{test11} and the fact that
 \begin{equation}\label{def23}\d G_1'(t) -\frac{\rho'(t)}{\rho(t)}G_1(t)= G_2(t),\end{equation}
 the equation  \eqref{eq6} yields
\begin{equation}
\begin{array}{l}\label{eq5bis}
\d G_2(t)+\left(\frac{\mu}{1+t}-\frac{\rho'(t)}{\rho(t)}\right)G_1(t)\\
=\d \int_0^t\int_{\R^N}\left\{|u_t(x,s)|^p+|u(x,s)|^q\right\}\psi(x,s)dx \, ds +\e \, C(f,g).
\end{array}
\end{equation}
Differentiating the  equation \eqref{eq5bis} in time and ignoring the nonnegative term in the right-hand side of the obtained equation, we get
\begin{align}\label{F1+bis}
\d G_2'(t)+\left(\frac{\mu}{1+t}-\frac{\rho'(t)}{\rho(t)}\right)G'_1(t)-\left(\frac{\mu}{(1+t)^2}+\frac{\rho''(t)\rho(t)-(\rho'(t))^2}{\rho^2(t)}\right)G_1(t) \ge 0.
\end{align}
Using  \eqref{lambda} and   \eqref{def23}, the identity \eqref{F1+bis} becomes
\begin{align}\label{F1+bis2}
\d G_2'(t)+\left(\frac{\mu}{1+t}-\frac{\rho'(t)}{\rho(t)}\right)G_2(t)-G_1(t) \ge 0.
\end{align}
Remember the definition of $\Gamma(t)$, given by \eqref{gamma}, we obtain 
\begin{equation}\label{G2+bis3}
\begin{array}{c}
\d G_2'(t)+\frac{3\Gamma(t)}{4}G_2(t)\ge\Sigma_1(t)+\Sigma_2(t),
\end{array}
\end{equation}
where 
\begin{equation}\label{sigma1-exp}
\Sigma_1(t)=\d \left(-\frac{\rho'(t)}{2\rho(t)}-\frac{\mu}{4(1+t)}\right)\left(G_2(t)+\left(\frac{\mu}{1+t}-\frac{\rho'(t)}{\rho(t)}\right)G_1(t)\right),
\end{equation}
and
\begin{equation}\label{sigma2-exp}
\Sigma_2(t)=\d \left(1+\left(\frac{\rho'(t)}{2\rho(t)}+\frac{\mu}{4(1+t)}\right) \left(\frac{\mu}{1+t}-\frac{\rho'(t)}{\rho(t)}\right) \right)  G_1(t).
\end{equation}

Now, using \eqref{eq5bis} and \eqref{lambda'lambda1}, we deduce that there exists $\tilde{T}_1=\tilde{T}_1(\mu) \ge T_0$ such that
\begin{equation}\label{sigma1}
\d \Sigma_1(t) \ge C \, \e, \quad \forall \ t \ge \tilde{T}_1. 
\end{equation}
Moreover, form Lemma \ref{F1} and \eqref{lambda'lambda1}, we conclude the existence of $\tilde{T}_2=\tilde{T}_2(\mu) \ge \tilde{T}_1(\mu)$ such that
\begin{equation}\label{sigma2}
\d \Sigma_2(t) \ge C \, \e, \quad \forall \ t  \ge  \tilde{T}_2. 
\end{equation}
Combining \eqref{G2+bis3}, \eqref{sigma1} and \eqref{sigma2}, we obtain
\begin{equation}\label{G2+bis4}
\begin{array}{c}
\d G_2'(t)+\frac{3\Gamma(t)}{4}G_2(t)\ge C \, \e, \quad \forall \ t  \ge  \tilde{T}_2.
\end{array}
\end{equation}
Multiplying  \eqref{G2+bis4} by $\frac{(1+t)^{3\mu/4}}{\rho^{3/2}(t)}$ and integrating over $(\tilde{T}_2,t)$, we deduce  that
\begin{align}\label{est-G111}
 G_2(t)
\ge G_2(\tilde{T}_2)\frac{\rho^{3/2}(t)(1+\tilde{T}_2)^{3\mu/4}}{\rho^{3/2}(\tilde{T}_2)(1+t)^{3\mu/4}}+C\,{\e}\frac{\rho^{3/2}(t)}{(1+t)^{3\mu/4}}\int_{\tilde{T}_2}^t\frac{(1+s)^{3\mu/4}}{\rho^{3/2}(s)}ds, \quad \forall \ t  \ge  \tilde{T}_2.
\end{align}
Now, observe that $G_2(t)=\rho(t)e^{t}F_2(t)$ where $F_2(t)$ is given by (3.4)  in \cite{Our}. Hence, using \cite[Lemma 3.3]{Our}\footnote{ The same arguments in the beginning of Remark \ref{rem3.2} hold true for the functional $F_2(t)$, given by (3.4)  in \cite{Our}, which is associated with only one derivative nonlinearity in the present work.} we infer that $G_2(t) \ge 0$ for all $t \ge 0$.\\
Therefore, using the above observation and  \eqref{lmabdaK}, we deduce that 
\begin{align}\label{est-G1111}
 G_2(\tilde{T}_2)\frac{\rho^{3/2}(t)(1+\tilde{T}_2)^{3\mu/4}}{\rho^{3/2}(\tilde{T}_2)(1+t)^{3\mu/4}}
\ge 0, \quad \forall \ t \ge 0.
\end{align}
Employing \eqref{est-double} and \eqref{est-G1111}, the estimate \eqref{est-G111} yields, for all $t \ge 2\tilde{T}_2$,
\begin{equation}\label{est-G2-12}
 G_2(t)
\ge   C\,{\e}e^{-3t/2} \int^t_{t/2}e^{3s/2}ds.
\end{equation}
Hence, we have
\begin{align}\label{est-G1-2}
 G_2(t)
\ge  C\,{\e}, \quad \forall \ t \ge T_1:=2\tilde{T}_2.
\end{align}

 This concludes the proof of Lemma
\ref{F11}.
\end{proof}

\begin{rem}\label{rem3.1}
Notice that in the proofs of Lemmas \ref{F1} and \ref{F11} we only used the positivity of each one of the two nonlinearities ($|u_t|^p$ and $|u|^q$). Indeed, the results in these lemmas are based on the comprehension of the dynamics in the linear part and, thus, the same conclusions can be handled similarly for any positive nonlinearity of the form $F(u,u_t)$ instead of $|u_t|^p+|u|^q$. 
\end{rem}

\begin{rem}\label{rem3.2}
Obviously,  the  results of Lemmas \ref{F1} and \ref{F11} naturally hold true when we consider one nonlinearity $|u_t|^p$ or $|u|^q$ as it is the case for  (\ref{T-sys-bis}). However,  to prove Theorem \ref{th_u_t}, and due to the nature of the equation, we make use of the linear part of (\ref{T-sys-bis}) together with the nonlinear term while estimating $G_2$. The approach used to estimate $G_2$ in the proof of Theorem \ref{th_u_t} is inspired from the computations carried out in \cite{LT2} for $F_2(t)=G_2(t) /(e^{t} \rho(t))$.
\end{rem}
\section{Proof of Theorem \ref{blowup}}\label{proof}
\par\quad
In this section, we will give the proof of the first (main) theorem in this article which states the blow-up result and the  lifespan estimate of the solution of (\ref{T-sys}). For that purpose, we will make use of the lemmas proven in Section \ref{aux} and a Kato's lemma type.

First, using the hypotheses in Theorem \ref{blowup}, we recall that $\mbox{\rm supp}(u)\ \subset\{(x,t)\in\R^N\times[0,\infty): |x|\le t+R\}$. \\
Let $t \in [0,T)$. Then, we set
\begin{equation}
F(t):=\int_{\R^N}u(x,t)dx,
\end{equation}
and 
\begin{equation}
\label{G-exp}
G(t):=\zeta(t)F(t) \ \text{with} \ \d \zeta(t)=(1+t)^{\frac{\mu}{2}}.
\end{equation}
Now, by choosing the test function $\Phi$ in \eqref{energysol2} such that
$\Phi\equiv 1$ in $\{(x,s)\in \R^N\times[0,t]:|x|\le s+R\}$\footnote{ The choice of a test function $\Phi$  which is identically equal to $1$ is possible thanks to the fact that the initial data $f$ and $g$ are supported on $B_{\R^N}(0,R)$.}, we get
\begin{equation}
\label{energysol1}
\begin{array}{r}
\d \d\int_{\R^N}u_t(x,t)dx-\int_{\R^N}u_t(x,0)dx + \int_0^t  \frac{\mu}{1+s}\int_{\R^N}u_t(x,s)dx ds
\\=\d \int_0^t  \int_{\R^N}\left\{|u_t(x,s)|^p+|u(x,s)|^q\right\}dx \,ds.
\end{array}
\end{equation}
Using the definition of $F(t)$, the equation \eqref{energysol1} can be written as
\begin{equation}
\label{F'0ineq12}
F'(t)+ \int_0^t  \frac{\mu}{1+s}F'(s) ds= F'(0)+ \int_0^t \int_{\R^N}\left\{|u_t(x,s)|^p+|u(x,s)|^q\right\}dx \,ds.
\end{equation}
Differentiating \eqref{F'0ineq12} in time, we obtain
\begin{equation}
\label{F'0ineq1-bis}
F''(t)+  \frac{\mu}{1+t}F'(t) = \int_{\R^N}\left\{|u_t(x,t)|^p+|u(x,t)|^q\right\}dx.
\end{equation}
Now, we introduce the following multiplier
\begin{equation}
\label{test1}
\begin{aligned}
m(t):=(1+t)^{\mu}. 
\end{aligned}
\end{equation}
Multiplying \eqref{F'0ineq1-bis} by $m(t)$ and integrating over $(0,t)$, we infer that
\begin{equation}
\label{F'0ineq11}
m(t)F'(t)= F'(0)+ \int_0^t m(s)\int_{\R^N}\left\{|u_t(x,s)|^p+|u(x,s)|^q\right\}dx \,ds.
\end{equation}
Therefore, by dividing \eqref{F'0ineq11} by $m(t)$, integrating over $(0,t)$  and using the positivity of $ F(0)$ and $F'(0)$, we deduce that
\begin{align}
\label{F'0ineqmmint}
F(t)\geq \int_0^t\frac{1}{m(s)} \int_0^sm(\tau ) \int_{\R^N}\left\{|u_t(x,\tau)|^p+|u(x,\tau)|^q\right\} dx \,d\tau\,ds.
\end{align}
By H\"{o}lder's inequality and the estimates \eqref{psi} and \eqref{F2postive}, we can bound the nonlinear term as follows:
\begin{equation}
\begin{array}{rcl}
\d \int_{\R^N}|u_t(x,t)|^pdx &\geq& \d G_2^p(t)\left(\int_{|x|\leq t+R}\Big(\psi(x,t)\Big)^{\frac{p}{p-1}}dx\right)^{-(p-1)} \vspace{.2cm}\\  &\geq&  C\rho^{-p}(t)e^{-pt}\e^p(1+t)^{ -\frac{(N-1)(p-2)}2}, \quad \forall \ t \ge T_1.
\end{array}
\end{equation}
Using \eqref{lmabdaK} and \eqref{est-double}, we get
 \begin{equation}\label{pho-est}
 \d \rho(t)e^{t} \le C (1+t)^{\frac{\mu}{2}}, \ \forall \ t \ge T_0/2.
 \end{equation}
Consequently, we have
\begin{equation}
\d \int_{\R^N}|u_t(x,t)|^pdx \geq C \e^p(1+t)^{ -\frac{\mu p+(N-1)(p-2)}2}, \ \forall \ t \ge T_1.\\
\end{equation}
Plugging the above inequality  into \eqref{F'0ineqmmint}, we obtain
\begin{equation}
\begin{aligned}\label{F0first}
F(t)
&\geq C\e^p (1+t)^{2 -\frac{\mu p+(N-1)(p-2)}2}, \ \forall \ t \ge T_1.
\end{aligned}
\end{equation}
Hence, by \eqref{G-exp}, the estimate \eqref{F0first} implies that
\begin{equation}
\begin{aligned}\label{F0first-1}
G(t)
&\geq C\e^p (1+t)^{2 -\frac{\mu (p-1)+(N-1)(p-2)}2}, \ \forall \ t \ge T_1.
\end{aligned}
\end{equation}

On the other hand, we have
\begin{align}
\Big(\int_{\R^N}u(x,t)dx\Big)^q\le \int_{|x|\le t+R}|u(x,t)|^qdx  \Big(\int_{|x|\le t+R}dx\Big)^{q-1}, 
\end{align}
and consequently we deduce that
\begin{align}\label{f0qsup}
G^q(t)\le  C\big(t+1 \big)^{N(q-1)+\frac{\mu q}{2}}  \int_{|x|\le t+R}|u(x,t)|^q dx.
\end{align}
Now, by differentiating  \eqref{F'0ineq11} with respect to time, we obtain 
\begin{equation}
\label{F'0ineq1}
(m(t)F'(t))' =   m(t)\int_{\R^N}\left\{|u_t(x,t)|^p+|u(x,t)|^q\right\}dx \ge   m(t)\int_{\R^N}|u(x,t)|^qdx.
\end{equation}
 Combining \eqref{f0qsup} in \eqref{F'0ineq1} and dividing the obtained equation (from \eqref{F'0ineq1}) by $\zeta(t)=\sqrt{m(t)}$, we infer that
\begin{equation}
\label{F'0ineq2}
G''(t) +\frac{\mu(2-\mu)}{4(1+t)^2} G(t)\ge  C\frac{G^q(t)}{\big(1+t \big)^{(N+\frac{\mu}{2})(q-1)}}, \ \forall \ t > 0.
\end{equation}

At this level, we distinguish two cases depending on the value of the parameter $\mu$. \\

{\bf First case ($\mu \ge 2$).} 

For this value of $\mu$, the estimate \eqref{F'0ineq2} yields
\begin{equation}
\label{F'0ineq2-1}
G''(t) \ge  C\frac{G^q(t)}{\big(1+t \big)^{(N+\frac{\mu}{2})(q-1)}}, \ \forall \ t > 0.
\end{equation}
Thanks to the fact that $G(t)=\zeta(t)F(t)$, \eqref{F'0ineq11} and \eqref{F'0ineqmmint} we have $G'(t)>0$. Then, multiplying \eqref{F'0ineq2-1} by $G'(t)$ implies that
\begin{equation}
\label{F25nov3}
\left\{\Big(G'(t)\Big)^2\right\}'
\ge C\frac{\Big(G^{q+1}(t)\Big)'}{(1+t)^{(N+\frac{\mu}{2})(q-1)}}, \ \forall \ t > 0.
\end{equation}
Integrating the above inequality, we obtain
\begin{equation}
\label{F25nov4-1}
\Big(G'(t)\Big)^2
\ge C\frac{G^{q+1}(t)}{(1+t)^{(N+\frac{\mu}{2})(q-1)}} +\left((G'(0))^2-CG^{q+1}(0)\right), \ \forall \ t > 0.
\end{equation}
Observe that the last term in the right-hand side of \eqref{F25nov4-1} is positive since we consider here small initial data, and more precisely this holds for $\e$ small enough.\\
Hence, \eqref{F25nov4-1} yields
\begin{equation}
\label{F25nov6}
\frac{G'(t)}{G^{1+\delta}(t)}
\ge C \frac{G^{\frac{q-1}2-\delta}(t)}{(1+t)^{\frac{(2N+\mu)(q-1)}{4}}}, \ \forall \ t > 0,
\end{equation}
for $\delta>0$ small enough.\\

{\bf Second case ($\mu < 2$).}

First, we recall that, as observed above, we have $G'(t)>0$. Then, multiplying \eqref{F'0ineq2} by $(1+t)^2 G'(t)$ gives
\begin{equation}
\label{F25nov3}
\frac{(1+t)^2}{2}\left(\left(G'(t)\right)^2\right)' + \frac{\mu(2-\mu)}{8} \left(G^2(t)\right)'
\ge C\frac{\Big(G^{q+1}(t)\Big)'}{\big(1+t \big)^{(N+\frac{\mu}{2})(q-1)-2}}, \ \forall \ t > 0.
\end{equation}
Integrating the above inequality and using the fact that $t \mapsto 1/\big(1+t \big)^{(N+\frac{\mu}{2})(q-1)-2}$ is a decreasing function (since $\d N(q-1)-2>0$ because  $q>1+\frac{2}{N}$ since we consider here $q>q_S(N+\mu)$\footnote{  It is clear that if $q \le q_S(N+\mu)$ the blow-up result can be proven by only considering the nonlinearity $|u(x,s)|^q$.}), we have
\begin{equation}
\label{F25nov4}
\begin{array}{c}
\d \frac{(1+t)^2}{2}\left(G'(t)\right)^2 + \frac{\mu(2-\mu)}{8} G^2(t)
\ge C_1\d \frac{G^{q+1}(t)}{\big(1+t \big)^{(N+\frac{\mu}{2})(q-1)-2}}+\vspace{.2cm} \\+\d G^2(0)\left(\frac{\mu(2-\mu)}{8}- C G^{q-1}(0)\right), \ \forall \ t > 0.
\end{array}
\end{equation}
Observe that the last term in the right-hand side of \eqref{F25nov4} is positive since we consider here small initial data, and more precisely this holds for $\e$ small enough. Hence, we have
\begin{equation}
\label{F25nov5}
\d \frac{(1+t)^2}{2}\left(G'(t)\right)^2 + \frac{\mu(2-\mu)}{8} G^2(t)
\ge C_1\d \frac{G^{q+1}(t)}{\big(1+t \big)^{(N+\frac{\mu}{2})(q-1)-2}}.
\end{equation}
We now aim to get rid of the second term in the left-hand side of \eqref{F25nov5}. Indeed,  in the present case $\mu < 2$, we will show that $ \frac{\mu(2-\mu)}{8} G^2(t)$ can be absorbed by the term in the right-hand side of \eqref{F25nov5}, namely $C_1\d G^{q+1}(t)/\big(1+t \big)^{(N+\frac{\mu}{2})(q-1)-2}$. In other words, our target consists in obtaining the following estimate:
\begin{equation}
\label{F25nov6}
\frac{G^{q-1}(t)}{\big(1+t \big)^{(N+\frac{\mu}{2})(q-1)-2}}>\frac{\mu(2-\mu)}{4C_1},
\end{equation}
for all $t\ge T_2$ where $T_2=T_2(\e,\mu)>0$ is given by \eqref{T2-choice} below.\\
Indeed,  employing the estimate  \eqref{F0first-1}, the expression of $G(t)$, the definition of $\lambda(p, q, N)$ (given by \eqref{1.5}) and the expression of $\zeta(t)$ (given by \eqref{G-exp}), we deduce that
\begin{equation}
\label{Gq-11}
\frac{G^{q-1}(t)}{\big(1+t \big)^{(N+\frac{\mu}{2})(q-1)-2}}>C_2  \e^{p(q-1)}(1+t)^{2-\frac{\lambda(p, q, N+\mu)}{2}}, \ \forall \ t \ge T_1.
\end{equation}
Now, we choose $T_2$ such that
\begin{equation}\label{T2-choice}
T_2= \max \left(C_3^{-\frac{2}{4-\lambda(p, q, N+\mu)}} \e^{-\frac{2p(q-1)}{4-\lambda(p, q, N+\mu)}},T_1(\mu)\right),
\end{equation}
where $\d C_3=4 C_1C_2/ (\mu(2-\mu))$. Note that for $\e$ small enough $$T_2=T_2(\e):=C_3^{-\frac{2}{4-\lambda(p, q, N+\mu)}} \e^{-\frac{2p(q-1)}{4-\lambda(p, q, N+\mu)}}.$$
Therefore \eqref{F25nov6} is now proven, and by combining \eqref{F25nov6} in \eqref{F25nov5}, we get
\begin{equation}
\label{F25nov7}
\d (1+t)^2\left(G'(t)\right)^2 
\ge C_1\d \frac{G^{q+1}(t)}{\big(1+t \big)^{(N+\frac{\mu}{2})(q-1)-2}}, \ \forall \ t\ge T_2,
\end{equation}
that we rewrite as follows:
\begin{equation}
\label{F25nov8}
\frac{G'(t)}{G^{1+\delta}(t)}
\ge C  \frac{G^{\frac{q-1}2-\delta}(t)}{(1+t)^{\frac{(2N+\mu)(q-1)}4}}, \ \forall \ t\ge T_2,
\end{equation}
for $\delta>0$ small enough. \\

Therefore, for the two cases $\mu \ge 2$ and $\mu < 2$ we end up with almost the same estimates (\eqref{F25nov6} and \eqref{F25nov8}, respectively); they only differ by the starting times $0$ and $T_2$, respectively. Hence, the  estimate \eqref{F25nov8} holds true for both cases for all $t\ge T_2$ where $T_2$ is given by \eqref{T2-choice}.

Now integrating the inequality \eqref{F25nov8} on $[t_1,t_2]$, for all $t_2>t_1 \ge T_2$, and using   \eqref{F0first-1}, we obtain
\begin{equation}\label{4.14}
\frac1{\delta}\Big (\frac{1}{G^{\delta}(t_1)}-\frac{1}{G^{\delta}(t_2)}\Big)
\ge C  (\e^p )^{\frac{q-1}2-\delta} \int_{t_1}^{t_2} \frac{(1+s)^{(2 -\frac{\mu (p-1)+(N-1)(p-2)}2)(\frac{q-1}2-\delta)}}{(1+s)^{\frac{(2N+\mu)(q-1)}4}}ds, \ \forall \ t_2>t_1 \ge T_2.
\end{equation}
Neglecting the second term in the left-hand side of \eqref{4.14} and using the definition of $\lambda(p, q, N)$ (given by \eqref{1.5}) yield
\begin{equation}\label{eq1/F}
\frac{1}{G^{\delta}(t_1)}
\ge C \delta \e^{\frac{p(q-1)}2-p\delta} \int_{t_1}^{t_2} (1+s)^{-\frac{ \lambda(p, q, N+\mu)}{4}-\delta\left(2 -\frac{\mu (p-1)+(N-1)(p-2)}2\right)}ds.  
\end{equation}
Employing the hypothesis \eqref{assump}, we have $-\frac{ \lambda(p, q, N+\mu)}{4}+1>0$. Hence, we  choose $\delta=\delta_0$ small enough such that $\gamma:=-\frac{ \lambda(p, q, N+\mu)}{4}-\delta_0\left(2 -\frac{\mu (p-1)+(N-1)(p-2)}2\right)>-1$. Then, the estimate \eqref{eq1/F} implies that
\begin{equation}\label{eq1/F-2}
\frac{1}{G^{\delta_0}(t_1)}
\ge C   \e^{\frac{p(q-1)}2-p\delta_0} \, \left((1+t_2)^{\gamma+1}-(1+t_1)^{\gamma+1}\right), \ \forall \ t_2>t_1 \ge T_2.
\end{equation}
Now, using \eqref{F0first-1},  we infer that
\begin{equation}\label{eq1/F-3}
\e^{\frac{p(q-1)}2} \, \left((1+t_2)^{\gamma+1}-(1+t_1)^{\gamma+1}\right) 
\le C_4 (1+t_1)^{-\delta_0\left(2 -\frac{\mu (p-1)+(N-1)(p-2)}2\right)}, \ \forall \ t_2>t_1 \ge T_2.
\end{equation}
Consequently, we have 
\begin{eqnarray}\label{eq1/F-4}
\e^{\frac{p(q-1)}2} \, (1+t_2)^{\gamma+1} 
&\le& C_4 (1+t_1)^{-\delta_0\left(2 -\frac{\mu (p-1)+(N-1)(p-2)}2\right)}\\&&+\e^{\frac{p(q-1)}2} \, (1+t_1)^{\gamma+1}, \ \forall \ t_2>t_1 \ge T_2.\nonumber
\end{eqnarray}
At this level, since $-\frac{ \lambda(p, q, N+\mu)}{4}+1>0$, then for all  $\e>0$, we choose $\tilde{T}_3$ such that
\begin{equation}\label{eq1/F-5}
\tilde{T}_3=C_4^{-\frac{4}{4-\lambda(p, q, N+\mu)}} \e^{-\frac{2p(q-1)}{4-\lambda(p, q, N+\mu)}}.
\end{equation}
Finally, we set $t_1=\max (T_2, \tilde{T}_3)$
Hence, using \eqref{eq1/F-5}, we deduce from \eqref{eq1/F-4} that
\begin{equation}\label{eq1/F-6}
t_2 \le 2^{\frac{1}{\gamma+1}}(1+t_1) \le C \e^{-\frac{2p(q-1)}{4-\lambda(p, q, N+\mu)}}.
\end{equation}

This achieves the proof of Theorem \ref{blowup}.\hfill $\Box$

\section{Proof of Theorem \ref{th_u_t}.}\label{sec-ut}

This section is devoted to the proof of Theorem \ref{th_u_t} which is somehow related to the determining  of the critical exponent associated with the nonlinear term in the problem \eqref{T-sys-bis}. First, we follow the computations already done in Section \ref{aux}  where we gain a better understanding of the linear  problem associated with \eqref{T-sys-bis} which is the same as  the linear problem associated with  \eqref{T-sys}. More precisely, we aim here to take advantage of the techniques used in Section \ref{aux}. We first note that Lemmas \ref{F1} and \ref{F11} remain true for the solution of \eqref{T-sys-bis} instead of \eqref{T-sys} (see Remark \ref{rem3.1} and the beginning of Remark \ref{rem3.2}) since we only used the positivity of the nonlinear terms and not their types. In fact, we proved in Lemma \ref{F11}  that $G_2(t)$ is a coercive function. This is a crucial observation that we will use to improve the blow-up result of  \eqref{T-sys-bis}. To this end, we use similar computations, as in \cite{LT2},  and the new  lower bound for the functional  $G_2(t)$ obtained in \eqref{F2postive}, namely the lower bound $\e / (1+t)^{\mu/2}$ in \cite[Lemma  4.1]{LT2} or  also in \cite[Lemma  3.3]{Our} is improved to reach $\e$ in the present work.

Thanks to the observation described above, we obtain an improvement of the blow-up result in \cite{Palmieri} (respec. \cite{LT2}) passing from   $p \in (1, p_G(N+\sigma)]$ where $p_G(N)$ is the Glassey exponent given by \eqref{Glassey} and $\sigma$ is given by \eqref{sigma} (respec.  $p_G(N+2\mu)$) to $p_G(N+\mu)$. More precisely, our result for (\ref{T-sys-bis}) enhances the corresponding one in  \cite{Palmieri}, for  $\mu \in (0,2)$, and coincides with it for $\mu \ge 2$.

We believe that the obtained critical exponent $p_G(N+\mu)$ may reach the threshold between the blow-up and the global existence regions.

In what follows we will use the equations and the estimates from \eqref{eq5bis} to \eqref{G2+bis3} with omitting the nonlinear term $|u(x,t)|^q$ and keeping the other one related to the nonlinearity $|u_t(x,t)|^p$. Hence, the analogous of the estimate \eqref{G2+bis3} reads 
\begin{equation}\label{G2+bis3-1}
\begin{array}{c}
\d G_2'(t)+\frac{3\Gamma(t)}{4}G_2(t)\ge\Sigma_1(t)+\Sigma_2(t)+\Sigma_3(t), \quad \forall \ t >0,
\end{array}
\end{equation}
where $\Sigma_1(t)$ and $\Sigma_2(t)$ are given, respectively, by \eqref{sigma1-exp} and \eqref{sigma2-exp}, and
\begin{equation}\label{sigma3-exp}
\Sigma_3(t):=  \int_{\R^N}|u_t(x,t)|^p\psi(x,t) dx.
\end{equation}
From \eqref{eq5bis} and \eqref{lambda'lambda1}, we have
\begin{equation}\label{sigma1}
\d \Sigma_1(t) \ge C \, \e+\left(-\frac{\rho'(t)}{2\rho(t)}-\frac{\mu}{4(1+t)}\right)\int_{0}^t \int_{\R^N}|u_t(x,s)|^p\psi(x,s)dx ds, \quad \forall \ t \ge \tilde{T}_1. 
\end{equation}
Now, using  \eqref{lambda'lambda1}, we deduce that there exists $\tilde{T}_3=\tilde{T}_3(\mu) \ge \tilde{T}_2$ such that
\begin{equation}\label{sigma3}
\d \Sigma_1(t) \ge C \, \e+\frac{1}{4}\int_{0}^t \int_{\R^N}|u_t(x,s)|^p\psi(x,s)dx ds, \quad \forall \ t  \ge \tilde{T}_3. 
\end{equation}
Plugging \eqref{sigma3} together with \eqref{sigma3-exp} and \eqref{sigma2} in \eqref{G2+bis3-1}, we deduce that
\begin{equation}\label{G2+bis5}
\begin{array}{rcl}
\d G_2'(t)+\frac{3\Gamma(t)}{4}G_2(t)&\ge& \d \frac{1}{4}\int_{0}^t \int_{\R^N}|u_t(x,s)|^p\psi(x,s)dx ds\vspace{.2cm}\\ &+&  \d \int_{\R^N}|u_t(x,t)|^p\psi(x,t) dx+C_5 \, \e, \quad \forall \ t  \ge \tilde{T}_3.
\end{array}
\end{equation}
Setting
\[
H(t):=
\frac{1}{8}\int_{\tilde{T}_4}^t \int_{\R^N}|u_t(x,s)|^p\psi(x,s)dx ds
+\frac{C_6 \e}{8},
\]
where $C_6=\min(C_5,8C_{G_2})$ ($C_{G_2}$ is defined in Lemma \ref{F11}) and $\tilde{T}_4>\tilde{T}_3$ is chosen such that $\frac{1}{4}-\frac{3\Gamma(t)}{32}>0$ and $\Gamma(t)>0$ for all $t \ge\tilde{T}_4$ (this is possible thanks to \eqref{gamma} and \eqref{lambda'lambda1}), and let
$$\mathcal{F}(t):=G_2(t)-H(t).$$
Hence, we have
\begin{equation}\label{G2+bis6}
\begin{array}{rcl}
\d \mathcal{F}'(t)+\frac{3\Gamma(t)}{4}\mathcal{F}(t) &\ge& \d \left(\frac{1}{4}-\frac{3\Gamma(t)}{32}\right)\int_{\tilde{T}_4}^t \int_{\R^N}|u_t(x,s)|^p\psi(x,s)dx ds\vspace{.2cm}\\ &+&  \d \frac{7}{8}\int_{\R^N}|u_t(x,t)|^p\psi(x,t) dx+C_6 \left(1-\frac{3\Gamma(t)}{32}\right) \e\\
&\ge&0, \quad \forall \ t \ge \tilde{T}_4.
\end{array}
\end{equation}
Multiplying  \eqref{G2+bis6} by $\frac{(1+t)^{3\mu/4}}{\rho^{3/2}(t)}$ and integrating over $(\tilde{T}_4,t)$, we deduce  that
\begin{align}\label{est-G111}
 \mathcal{F}(t)
\ge \mathcal{F}(\tilde{T}_4)\frac{(1+\tilde{T}_4)^{3\mu/4}}{\rho^{3/2}(\tilde{T}_4)}\frac{\rho^{3/2}(t)}{(1+t)^{3\mu/4}}, \ \forall \ t \ge \tilde{T}_4,
\end{align}
where $\rho(t)$ is defined by \eqref{lmabdaK}.\\
Therefore we have $\d \mathcal{F}(\tilde{T}_4)=G_2(\tilde{T}_4)-\frac{C_6 \e}{8} \ge G_2(\tilde{T}_4)-C_{G_2}\e \ge 0$ thanks to Lemma \ref{F11} and the fact that $C_6=\min(C_5,8C_{G_2}) \le 8C_{G_2}$. \\
Then, we have 
\begin{equation}
\label{G2-est}
G_2(t)\geq H(t), \ \forall \ t \ge \tilde{T}_4.
\end{equation}
By H\"{o}lder's inequality and the estimates \eqref{psi} and \eqref{F2postive}, we can bound the nonlinear term as follows:
\begin{equation}
\begin{array}{rcl}
\d \int_{\R^N}|u_t(x,t)|^p\psi(x,t)dx &\geq&\d G_2^p(t)\left(\int_{|x|\leq t+R}\psi(x,t)dx\right)^{-(p-1)} \vspace{.2cm}\\ &\geq& C G_2^p(t) \rho^{-(p-1)}(t)e^{-(p-1)t}(1+t)^{-\frac{(N-1)(p-1)}2}.
\end{array}
\end{equation}
Using \eqref{pho-est}, we obtain  
\begin{equation}
\d \int_{\R^N}|u_t(x,t)|^p\psi(x,t)dx \geq C G_2^p(t)(1+t)^{-\frac{(N+\mu-1)(p-1)}2}, \ \forall \ t \ge \tilde{T}_4.
\end{equation}
From the above estimate and \eqref{G2-est}, we infer that
\begin{equation}
\label{inequalityfornonlinearin}
H'(t)\geq C H^p(t)(1+t)^{-\frac{(N+\mu-1)(p-1)}2}, \quad \forall \ t \ge \tilde{T}_4.
\end{equation}
Since $H(\tilde{T}_4)=C_6 \e/8>0$, 
we can easily get the upper bound of the lifespan estimate in Theorem \ref{th_u_t}.

\section{Appendix}
In this appendix, we will recall some properties of the function $\rho(t)$, the solution of \eqref{lambda}. Hence, following the computations in \cite{Tu-Lin} (with $\eta=1)$, we can write the expression of  $\rho(t)$ as follows:
\begin{equation}\label{lmabdaK-A}
\rho(t)=(t+1)^{\frac{\mu+1}{2}}K_{\frac{\mu-1}2}(t+1),
\end{equation}
where 
$$K_{\nu}(t)=\int_0^\infty\exp(-t\cosh \zeta)\cosh(\nu \zeta)d\zeta,\ \nu\in \mathbb{R}.$$
Using the property of  $\rho(t)$ in the proof of Lemma 2.1 in \cite{Tu-Lin} (with $\eta=1)$, we infer that
\begin{equation}\label{lambda'lambda}
\frac{\rho'(t)}{\rho(t)}=\frac{\mu}{1+t}-\frac{K_{\frac{\mu+1}2}(t+1)}{K_{\frac{\mu-1}2}(t+1)}.
\end{equation}
From \cite{Gaunt}, we have the following property for the function $K_{\mu}(t)$:
\begin{equation}\label{Kmu}
K_{\mu}(t)=\sqrt{\frac{\pi}{2t}}e^{-t} (1+O(t^{-1}), \quad \text{as} \ t \to \infty.
\end{equation}
Combining \eqref{lambda'lambda} and \eqref{Kmu}, we infer that
\begin{equation}\label{lambda'lambda1}
\frac{\rho'(t)}{\rho(t)}=-1+O(t^{-1}), \quad \text{as} \ t \to \infty.
\end{equation}

Finally, we refer the reader to \cite{Erdelyi} for more details about the properties of the function $K_{\mu}(t)$.


\bibliographystyle{plain}

\end{document}